\documentclass[12 pt, reqno]{amsart}
\usepackage{amsmath,times,epsfig,amssymb,amsbsy,amscd,amsfonts,amstext,graphicx,color,bm,amsthm}
\usepackage[all]{xy}
\usepackage{graphicx}

\usepackage[urlcolor=blue]{hyperref}
\usepackage{tikz,subfigure}

\renewcommand{\to}{\longrightarrow}

\newtheorem{Theorem}{Theorem}[section]
\newtheorem{Definition}[Theorem]{Definition}

\newtheorem{Lemma}[Theorem]{Lemma}
\newtheorem{Proposition}[Theorem]{Proposition}
\newtheorem{Corollary}[Theorem]{Corollary}
\newtheorem{Remark}[Theorem]{Remark}
\newtheorem{Example}[Theorem]{Example}

\DeclareMathOperator{\Der}{Der}

\DeclareMathOperator{\pdeg}{pdeg}

\DeclareMathOperator{\coker}{coker}

\newcommand \KK {{\mathbb K}}
\def \X(#1){\{x_1,\dots, x_{#1}\}}

\newcommand{\A}{\mathcal{A}}

\newcommand{\Ker}{\operatorname{Ker}}

\newcommand \ideal[1] {\langle #1 \rangle}

\begin{document}

\title[Freeness for multiarrangements of hyperplanes]{Freeness for multiarrangements of hyperplanes over arbitrary fields}

\begin{abstract}
In this paper, we study the class of free multiarrangements of hyperplanes. Specifically, we investigate the relations between freeness over a field of finite characteristic and freeness over $\mathbb{Q}$.
\end{abstract}

\author{Michele Torielli}
\address{Michele Torielli, Department of Mathematics, GI-CoRE GSB, Hokkaido University, Kita 10, Nishi 8, Kita-Ku, Sapporo 060-0810, Japan.}
\email{torielli@math.sci.hokudai.ac.jp}


\date{\today}
\maketitle


\section{Introduction}
Let $V$ be a vector space of dimension $l$ over a field $\KK$. Fix a system of coordinate $(x_1,\dots, x_l)$ of $V^\ast$. 
We denote by $S = S(V^\ast) = \KK[x_1,\dots, x_l]$ the symmetric algebra. 
A hyperplane arrangement $\A = \{H_1, \dots, H_n\}$ is a finite collection of hyperplanes in $V$.

The theory of freeness of hyperplane arrangements is a key notion which connects arrangement theory with algebraic geometry and combinatorics. By definition, an arrangement is free if and only if its module of logarithmic derivations is a free module. A lot it is known about free arrangements, however there is still some mystery around the notion of freeness.
The notion of freeness was introduced by Saito in \cite{saito} for the case of hypersurfaces in the analytic category. The special case of hyperplane arrangements was firstly studied by Terao in \cite{terao1980arrangementsI}, where he showed that we can pass from analytic to algebraic considerations. In \cite{ziegler1986multiarrangements}, Ziegler extended this theory to the class of multiarrangements of hyperplanes, i.e. arrangements in which each hyperplane has a positive integer multiplicity. In addition, Ziegler put in relation the notion of free arrangements and the one of free multiarrangements. Since their introduction, free (multi)arrangements have been intensively studied in connection with the famous Terao's conjecture. See for example  \cite{yoshinaga2004characterization},\cite{abe2008euler}, \cite{hoge2014ziegler}, \cite{yoshinaga2014freeness} and \cite{Gin-freearr}.

The purpose of this paper is to extend the work of  \cite{palezzato2018free} in order study the connections between freeness of multiarrangements over a field of characteristic zero and over a finite field, and to describe in which cases the two situations are related and how.

\section{Preliminares on hyperplane arrangements}\label{sec:arr}

In this section, we recall the terminology, the basic notations and some fundamental results related to multiarrangements of hyperplanes.

Let $\KK$ be a field. A finite set of affine hyperplanes $\A =\{H_1, \dots, H_n\}$ in $\KK^l$ 
is called a \textbf{hyperplane arrangement}. For each hyperplane $H_i$ we fix a defining equation $\alpha_i\in S= \KK[x_1,\dots, x_l]$ such that $H_i = \alpha_i^{-1}(0)$, 
and let $Q(\A)=\prod_{i=1}^n\alpha_i$. An arrangement $\A$ is called \textbf{central} if each $H_i$ contains the origin of $\KK^l$. 
In this case, the defining equation $\alpha_i\in S$ is linear homogeneous, and hence $Q(\A)$ is homogeneous of degree $n$. In this paper, we will only consider central arrangements.

\begin{Definition} A \textbf{multiarrangement of hyperplanes} is a pair $(\A, \bold{m})$ of a central arrangement $\A$ with a map $\bold{m}\colon \A\to \mathbb{Z}_{\ge0}$, called the \textbf{multiplicity}. We also put $Q(\A, \bold{m})=\prod_{i=1}^n\alpha_i^{\bold{m}(H_i)}$ and $|\bold{m}|=\sum_{i=1}^n\bold{m}(H_i)$.
\end{Definition} 

The theory of multiarrangements is a generalization of the one of arrangements. In fact, an arrangement $\A$ can be identified with $(\A,\bold{1})$ a multiarrangement with constant multiplicity $\bold{m}\equiv 1$, which is sometimes called a \textbf{simple arrangement}.

We denote by $\Der_{\KK^l} =\{\sum_{i=1}^l f_i\partial_{x_i}~|~f_i\in S\}$ the $S$-module of \textbf{polynomial vector fields} on $\KK^l$ (or $S$-derivations). 
Let $\delta =  \sum_{i=1}^l f_i\partial_{x_i}\in \Der_{\KK^l}$. Then $\delta$ is  said to be \textbf{homogeneous of polynomial degree} $d$ if $f_1, \dots, f_l$ are homogeneous polynomials of degree~$d$ in $S$. 
In this case, we write $\pdeg(\delta) = d$.

\begin{Definition} 
Let $(\A,\bold{m})$ be a multiarrangement in $\KK^l$. Define the \textbf{module of vector fields logarithmic tangent} to $\A$ with multiplicity $\bold{m}$ (or logarithmic vector fields) by
$$D(\A,\bold{m}) = \{\delta\in \Der_{\KK^l}~|~ \delta(\alpha_i) \in \ideal{\alpha_i^{\bold{m}(H_i)}}S, \forall i\}.$$
\end{Definition}

The module $D(\A,\bold{m})$ is obviously a graded $S$-module and plays a central role in the theory of free multiarrangements of hyperplanes. In general, in contrast to the case of simple arrangements, we have that
$D(\A,\bold{m})$ does not coincide with $\{\delta\in \Der_{\KK^l}~|~ \delta(Q(\A)) \in \ideal{ Q(\A,\bold{m})} S\}$. 

\begin{Definition} 
A multiarrangement $(\A, \bold{m})$ in $\KK^l$ is said to be \textbf{free with exponents} $(e_1,\dots,e_l)$
if and only if $D(\A, \bold{m})$ is a free $S$-module and there exists a basis $\delta_1,\dots,\delta_l \in D(\A, \bold{m})$ 
such that $\pdeg(\delta_i) = e_i$, or equivalently $D(\A, \bold{m})\cong\bigoplus_{i=1}^lS(-e_i)$.
\end{Definition}

As described in \cite{ziegler1986multiarrangements}, since $D(\A,\bold{m})$ is a reflexive $S$-module, then any multiarrangement in $\KK^2$ is free.

\begin{Example} Consider the multiarrangement $(\A, \bold{m})$ in $\mathbb{R}^2$ with $Q(\A, \bold{m})=x^2y^2(x-y)$. Then $(\A, \bold{m})$ is free with exponents $(2,3)$. In fact, $D(\A, \bold{m})$ is free and it is generated by 
$$\delta_1=x^2\frac{\partial}{\partial x}+y^2\frac{\partial}{\partial y},$$
$$\delta_2=(x-y)y^2\frac{\partial}{\partial y}.$$
\end{Example}

\begin{Remark}\label{rem:detdivisoblebyeq}
Consider $\delta_1,\dots,\delta_l \in D(\A, \bold{m})$, then $\det(\delta_i(x_j))_{i,j}$ is divisible by $Q(\A, \bold{m})$.
\end{Remark}

In \cite{saito}, after introducing the notion of freeness, Saito described one of the most famous characterizations of freeness,
see also \cite{ziegler1986multiarrangements}. Saito's criterion
checks if $(\A, \bold{m})$ is free or not by looking at the determinant of the coefficient 
matrix of $\delta_1,\dots,\delta_l\in D(\A, \bold{m})$. Notice that the original statement is for
characteristic zero. However, as noted in \cite{terao1983free}, this statement holds true for any characteristic.

\begin{Theorem}[Saito's criterion]\label{theo:saitocrit}
Let $(\A, \bold{m})$ be a multiarrangement in $\KK^l$ and $\delta_1, \dots, \delta_l \in D(\A, \bold{m})$. Then the following facts are equivalent
\begin{enumerate}
\item $D(\A,\bold{m})$ is free with basis $\delta_1, \dots, \delta_l$, i. e. $D(\A, \bold{m}) = S\cdot\delta_1\oplus \cdots \oplus S \cdot\delta_l$.
\item $\det(\delta_i(x_j))_{i,j}=c Q(\A, \bold{m})$, where $c\in \KK\setminus\{0\}$.
\item $\delta_1, \dots, \delta_l$ are linearly independent over $S$ and $\sum_{i=1}^l\pdeg(\delta_i)=|\bold{m}|$.
\end{enumerate}
\end{Theorem}

Directly from Saito's criterion, we obtain that if $(\A, \bold{m})$ is a free multiarrangement in $\KK^l$ with exponents $(e_1,\dots, e_l)$, then $\deg(Q(\A,\bold{m}))=|\bold{m}|=\sum_{i=1}^le_i$.

Notice that in contrast from the theory of simple arrangements, see \cite{orlterao}, the exponents of a free multiarrangement are not combinatorial in general, as shown in the following example.

\begin{Example}[\cite{ziegler1986multiarrangements}]\label{ex:specialexample} Consider the multiarrangement $(\A, \bold{m})$ in $\mathbb{R}^2$ with defining polynomial $Q(\A, \bold{m})=x^3y^3(x-y)(x+y)$. Then $(\A, \bold{m})$ is free with exponents $(3,5)$. However, if we consider the multiarrangement $(\A', \bold{m})$ in $\mathbb{R}^2$ with $Q(\A', \bold{m})=x^3y^3(x-y)(x+4y)$, then also $(\A', \bold{m})$ is free but with exponents $(4,4)$.
\end{Example}

Given a multiarrangement $(\A,\bold{m})$ in $\KK^l$ with $\A=\{H_1,\dots, H_n\}$, let $M(\A,\bold{m})$ be the $S$-submodule of the free $S$-module $S^n$ defined by
$$M(\A,\bold{m})=\langle(\alpha_1^{\bold{m}(H_1)},0,\dots,0), \dots, (0,\dots,0,\alpha_n^{\bold{m}(H_n)}) \rangle.$$
Consider the $n\times l$ matrix $A(\A)=(\frac{\partial\alpha_i}{\partial x_j})_{i,j}$ with coefficients in $S$. We can then define the multiplication map of $S$-modules 
\begin{equation}\label{eq:importantmapdef}
\varphi_{(\A,\bold{m})}\colon S^l\to S^n/M(\A,\bold{m})
\end{equation}
defined by $(g_1,\dots,g_l)^t\mapsto A(\A)(g_1,\dots,g_l)^t$.

With this construction, it is trivial to prove the following
\begin{Proposition}\label{prop:deraskernel} Let $(\A, \bold{m})$ be a multiarrangement in $\KK^l$. Then $$D(\A,\bold{m})\cong\Ker(\varphi_{(\A,\bold{m})}).$$
\end{Proposition}

\section{From characteristic $0$ to characteristic $p$}\label{sec:freefromchar0tocharp}

Similarly to \cite{palezzato2018free} and \cite{palezzato2019combinatorially}, from now on we will assume that $(\A,\bold{m})$ is a multiarrangement in $\mathbb{Q}^l$. After getting rid of the denominators, we can suppose that $\alpha_i \in\mathbb{Z}[x_1,\dots, x_l]$ for all $i=1,\dots, n$, and hence that $Q(\A,\bold{m})=\prod_{i=1}^n\alpha_i^{\bold{m}(H_i)}\in\mathbb{Z}[x_1,\dots, x_l]$. Moreover, we can also assume that there exists no prime number $p$ that divides any $\alpha_i$.

Let $p$ be a prime number. Consider the image of each $\alpha_i$ under the canonical homomorphism $\pi_p\colon \mathbb{Z}[x_1,\dots, x_l]\to \mathbb{F}_p[x_1,\dots, x_l]$. By assumption, $\pi_p(\alpha_i)\ne0$ for all $i=1,\dots, n$.

\begin{Definition} Let $(\A,\bold{m})$ be a multiarrangement in $\mathbb{Q}^l$. We will call a prime number $p$ \textbf{good} for $(\A, \bold{m})$ if $\pi_p(\alpha_i)\ne\pi_p(\alpha_j)$ for all $i\ne j$.
\end{Definition}

Similarly to the case of simple arrangements, described in \cite{palezzato2018free}, we have the following.

\begin{Lemma}\label{lem:finitenotgoodprimes} There is a finite number of primes $p$ that are not good for $(\A,\bold{m})$.
\end{Lemma}
\begin{proof} By definition, $p$ is not good for $(\A,\bold{m})$ if and only if $\pi_p(\alpha_i)=\pi_p(\alpha_j)$ for some $i\ne j$ and this can happens only for a finite number of primes.
\end{proof}

Let now $p$ be a good prime for $(\A, \bold{m})$, and consider $(\A_p, \bold{m})$ the multiarrangement in $\mathbb{F}_p^l$ defined by $\pi_p(Q(\A, \bold{m}))$. By construction, $Q(\A_p, \bold{m})=\pi_p(Q(\A,\bold{m}))=\prod_{i=1}^n\pi_p(\alpha_i)^{\bold{m}(H_i)}\ne0$ and it has degree $|\bold{m}|$.

\begin{Theorem}\label{theo:fromchar0tocharP} If $(\A,\bold{m})$ is a free multiarrangement in $\mathbb{Q}^l$ with exponents $(e_1,\dots, e_l)$, then $(\A_p,\bold{m})$ is a free multiarrangement in $\mathbb{F}_p^l$ with exponents $(e_1,\dots, e_l)$, 
for all good primes except possibly a finite number of them.
\end{Theorem}
\begin{proof} Let $\Delta=\{\delta_1,\dots, \delta_l\}$ be a basis of $D(\A,\bold{m})$ with $\pdeg(\delta_i)=e_i$ for all $i=1,\dots, l$. Since $Q(\A,\bold{m})$ has only integer coefficients, 
we can assume that every polynomial that appears in each $\delta\in\Delta$ is in $\mathbb{Z}[x_1,\dots, x_l]$. Hence we can consider $\bar{\delta}_1,\dots, \bar{\delta}_l\in \Der_{\mathbb{F}_p^l}$ 
 the image of $\delta_1,\dots, \delta_l$. 
We can assume that $p\nmid\delta$ for all $\delta\in\Delta$, and hence $\bar{\delta}\ne0$ for all $\delta\in\Delta$. This implies that $\pdeg(\bar{\delta}_i)=\pdeg(\delta_i)=e_i$ for all $i=1,\dots, l$.

Fix one $j=1,\dots, n$. Then by the definition of $D(\A,\bold{m})$, for each $\delta\in \Delta$ there exists $h_j\in\mathbb{Z}[x_1,\dots, x_l]$ such that $\delta(\alpha_j)=h_j\alpha_j^{\bold{m}(H_j)}$. If we apply $\pi_p$ to this expression we obtain that $\bar{\delta}(\pi_p(\alpha_j))=\pi_p(\delta(\alpha_j))=\pi_p(h_j\alpha_j^{\bold{m}(H_j)})=\pi_p(h_j)\pi_p(\alpha_j)^{\bold{m}(H_j)}$. Since this holds for all $\alpha_j$, then $\bar{\delta}\in D(\A_p,\bold{m})$.

By Theorem~\ref{theo:saitocrit}, since in each $\delta\in\Delta$  every polynomial that appears have only integer coefficients, there exists $c\in\mathbb{Z}\setminus\{0\}$ such that $\det(\delta_i(x_j))_{i,j}=c Q(\A, \bold{m})$. 
If we apply $\pi_p$ to the previous equality we get that $\det(\bar{\delta}_i(x_j))_{i,j}=\pi_p(\det(\delta_i(x_j))_{i,j})=\pi_p(c Q(\A,\bold{m}))=\pi_p(c)Q(\A_p,\bold{m})$. Hence if $p$ does not divide $c$, 
we have that $\pi_p(c)\in\mathbb{F}_p\setminus\{0\}$ and hence again by Theorem \ref{theo:saitocrit}, we have that $\bar{\delta}_1,\dots, \bar{\delta}_l$ are a basis of $D(\A_p,\bold{m})$. 
This proves that $(\A_p,\bold{m})$ is free with exponents $(e_1,\dots, e_l)$.
\end{proof}

By Lemma \ref{lem:finitenotgoodprimes}, the number of non-good primes is finite. Hence we have the following.
\begin{Corollary}\label{corol:fromchar0tocharP} Let $(\A,\bold{m})$ be a multiarrangement in $\mathbb{Q}^l$ and $p$ a large prime number. If $(\A,\bold{m})$ is free in $\mathbb{Q}^l$ with exponents $(e_1,\dots, e_l)$, then $(\A_p,\bold{m})$ is free in $\mathbb{F}_p^l$ with exponents $(e_1,\dots, e_l)$.
\end{Corollary}

If we consider the second multiarrangement of Example \ref{ex:specialexample} over a different field, we have the following.
\begin{Example} Consider the multiarrangement $(\A',\bold{m})$ of Example \ref{ex:specialexample} with $\KK=\mathbb{Q}$. All prime numbers $p\ne2$ are good for $(\A',\bold{m})$. A direct computation shows that
$(\A',\bold{m})$ is free with exponents $(4,4)$. However, if we consider $((\A')_3,\bold{m})$ as a multiarrangement in $\mathbb{F}_3^2$, then it is still free but its exponents are $(3,5)$.
This is because if we can consider the matrix
$$\left(\begin{array}{cc} x^4 & 13xy^3-12y^4\\ x^3y & -3xy^3 +4y^4 \end{array}\right) $$
as the coefficient matrix of a basis of $D(\A',\bold{m})$, its determinant is equal to
$-3Q(\A',\bold{m})$. A direct computation shows that if we take another basis of $D(\A',\bold{m})$ with only integer coefficients, then the determinant of the coefficient matrix is equal to $cQ(\A',\bold{m})$ with $c\in 3\mathbb{Z}\setminus\{0\}$.
This is why over $\mathbb{F}_3$ the exponents of $(\A',\bold{m})$ change.
\end{Example}


 The following is an example of a free multiarrangements in $\mathbb{Q}^l$ that is not free in $\mathbb{F}_p^l$, for some good prime $p$.
 \begin{Example} Consider the multiarrangement $(\A,\bold{m})$ in $\mathbb{Q}^3$ with defining polynomial $Q(\A,\bold{m})=x^2y^2z^2(x-y)^2(x-z)^2(y-z)^2$. In this situation, all prime numbers are good for $(\A,\bold{m})$, and
 $(\A,\bold{m})$ is free with exponents $(4,4,4)$. However, if we consider $(\A_2,\bold{m})$ as a multiarrangement in $\mathbb{F}_2^3$, then $(\A_2,\bold{m})$ is free with exponents $(2,4,6)$. On the other hand, if we consider $(\A_3,\bold{m})$ as a multiarrangement in $\mathbb{F}_3^3$, then $(\A_3,\bold{m})$ is not free.
 This is because the determinant of a coefficient matrix of a basis of $D(\A,\bold{m})$ is equal to $cQ(\A,\bold{m})$ with $c\in 18\mathbb{Z}\setminus\{0\}$. In particular, if we consider $p\ne2,3$ a prime number and $(\A_p,\bold{m})$ as a multiarrangement in $\mathbb{F}_p^3$, then $(\A_p,\bold{m})$ is free with exponents $(4,4,4)$.
 \end{Example}

\section{From characteristic $p$ to characteristic $0$}\label{sec:fromPtozerochar}
As in Section \ref{sec:freefromchar0tocharp}, we will assume that $(\A,\bold{m})$ is a multiarrangement in $\mathbb{Q}^l$, that $\alpha_i \in\mathbb{Z}[x_1,\dots, x_l]$ for all $i=1,\dots, n$, and that there exists no prime number $p$ that divides any $\alpha_i$. Moreover, let $S=\mathbb{Q}[x_1,\dots, x_l]$, $S_\mathbb{Z}=\mathbb{Z}[x_1,\dots, x_l]$ and $S_p=\mathbb{F}_p[x_1,\dots, x_l]$. 

Similarly to the construction of the map~\eqref{eq:importantmapdef} at the end of Section \ref{sec:arr}, by our assumptions on the $\alpha_i$, we can consider $M(\A,\bold{m})$ as $S_\mathbb{Z}$-submodule of $S_\mathbb{Z}^n$, and $A(\A)$ as matrix with coefficients in $S_\mathbb{Z}$. Hence we can construct the map of  $S_\mathbb{Z}$-modules
\begin{equation}\label{eq:importantmapdefZ}
\varphi_{\mathbb{Z}}\colon S_\mathbb{Z}^l\to S_\mathbb{Z}^n/M(\A,\bold{m})
\end{equation}
defined by $(g_1,\dots,g_l)^t\mapsto A(\A)(g_1,\dots,g_l)^t$. It is trivial to see that
\begin{Lemma}\label{lemma:DerfromZtoQ} Let $(\A, \bold{m})$ be a multiarrangement in $\mathbb{Q}^l$. Then $$D(\A,\bold{m})\cong\Ker(\varphi_\mathbb{Z})\otimes_{\mathbb{Z}}\mathbb{Q}.$$
\end{Lemma}
For any integer $k\ge1$, the canonical homomorphism $\pi_p\colon S_\mathbb{Z}\to S_p$ extends naturally to the homomorphism $\pi_p^k\colon S_\mathbb{Z}^k\to S^k_p$. If we assume that $p$ is a good prime for $(\A, \bold{m})$, then $\pi_p^n(M(\A,\bold{m}))=M(\A_p,\bold{m})$. This implies that we can construct the following commutative diagram

\begin{equation}
\xymatrix{\Ker(\varphi_\mathbb{Z}) \ar[d]^{\pi_p^l}  \ar@{^{(}->}[r]^-i &S_\mathbb{Z}^l \ar[d]^{\pi_p^l} \ar[r]^-{\varphi_{\mathbb{Z}}}& S_\mathbb{Z}^n/M(\A,\bold{m}) \ar[d]^{\pi_p^n}\\
D(\A_p,\bold{m}) \ar@{^{(}->}[r]^-{i_p} &S_p^l \ar[r]^-{\varphi_{(\A_p,\bold{m})}}& S_p^n/M(\A_p,\bold{m}) } \label{diagmapsZFp} \end{equation}

The introduction of the commutative diagram~\eqref{diagmapsZFp} allows us to describe in which situations the freeness of a multiarrangement over $\mathbb{F}_p$ implies the freeness over $\mathbb{Q}$.
\begin{Theorem}\label{theo:fromPtozero} Let $(\A,\bold{m})$ be a multiarrangement in $\mathbb{Q}^l$. Let $p$
be a good prime number for $(\A,\bold{m})$ and assume that the map 
$$\pi_p^l\colon \Ker(\varphi_\mathbb{Z})\to D(\A_p,\bold{m})$$ 
is surjective. If $(\A_p,\bold{m})$ is free in $\mathbb{F}_p^l$ with exponents $(e_1,\dots, e_l)$, 
then $(\A,\bold{m})$ is free in $\mathbb{Q}^l$ with exponents $(e_1,\dots, e_l)$.
\end{Theorem}
\begin{proof} Let $\delta_1,\dots,\delta_l$ be a basis of $D(\A_p,\bold{m})$ such that $\pdeg(\delta_i)=e_i$. Since the map $\pi_p^l\colon \Ker(\varphi_\mathbb{Z})\to D(\A_p,\bold{m})$ is surjective, there exist $\tilde{\delta}_1,\dots,\tilde{\delta}_l$ in $\Ker(\varphi_\mathbb{Z})\setminus\{0\}$ such that $\pi_p^l(\tilde{\delta}_i)=\delta_i$. We can assume that each $\tilde{\delta}_i$ is homogeneous. Clearly $\pdeg(\tilde{\delta}_i)=e_i$.

By Lemma~\ref{lemma:DerfromZtoQ}, we can consider $\tilde{\delta}_1,\dots,\tilde{\delta}_l$ as elements of $D(\A,\bold{m})$. By Remark~\ref{rem:detdivisoblebyeq}, we have that $\det(\tilde{\delta}_i(x_j)_{i,j})=hQ(\A,\bold{m})$ for some $h\in S_\mathbb{Z}\subseteq S$. On the other hand,
$\deg(\det(\tilde{\delta}_i(x_j)_{i,j}))=\sum_{k=1}^l\pdeg(\tilde{\delta}_k)=\sum_{k=1}^le_k=|\bold{m}|=\deg(Q(\A,\bold{m}))$. This implies that $h\in\mathbb{Z}\subseteq\mathbb{Q}$. Suppose that $h=0$, then $\det(\tilde{\delta}_i(x_j)_{i,j})=0$. If we apply the map $\pi_p\colon S_\mathbb{Z}\to S_p$ to this equality we obtain
$0=\pi_p(\det(\tilde{\delta}_i(x_j)_{i,j}))=\det(\delta_i(x_j)_{i,j})$. However this is impossible, since, by Theorem~\ref{theo:saitocrit}, $\det(\delta_i(x_j)_{i,j})=cQ(\A_p,\bold{m})$ for some $c\in\mathbb{F}_p\setminus\{0\}$. This implies that $\det(\tilde{\delta}_i(x_j)_{i,j})=hQ(\A,\bold{m})$ for some $h\in\mathbb{Q}\setminus\{0\}$. If we apply Theorem~\ref{theo:saitocrit} to $\tilde{\delta}_1,\dots,\tilde{\delta}_l$, we obtain that they form a basis of $D(\A,\bold{m})$, and hence that $\A$ is free with exponents $(e_1,\dots, e_l)$.
\end{proof}

In general, it might not be easy to check directly the surjectivity of the map $\pi_p^l\colon \Ker(\varphi_\mathbb{Z})\to D(\A_p,\bold{m})$. However, we can obtain information by looking at $\coker(\varphi_\mathbb{Z})$.

\begin{Proposition}\label{prop:zerodivtosurg} Let $(\A, \bold{m})$ be a multiarrangement in $\mathbb{Q}^l$ and $p$ a good prime for $(\A, \bold{m})$. If $p$ is not a zero divisor of $\coker(\varphi_\mathbb{Z})$, then the map $\pi_p^l\colon \Ker(\varphi_\mathbb{Z})\to D(\A_p,\bold{m})$ is surjective.
\end{Proposition}
\begin{proof} Assume by absurd that there exists $\delta\in D(\A_p,\bold{m})\setminus\pi_p^l(\Ker(\varphi_\mathbb{Z}))$. Since the map $\pi_p^l\colon S_\mathbb{Z}^l\to S_p^l$ is surjective, there exists $\tilde{\delta}\in S_\mathbb{Z}^l\setminus\Ker(\varphi_\mathbb{Z})$ such that $\pi_p^l(\tilde{\delta})=i_p(\delta)$. Since $i_p(\delta)\ne0$, then $p$ does not divide $\tilde{\delta}$. By construction, $\varphi_\mathbb{Z}(\tilde{\delta})\ne0$ and $\pi_p^n(\varphi_\mathbb{Z}(\tilde{\delta}))=\varphi_{(\A_p,\bold{m})}(i_p(\delta))=0$. This implies that $\varphi_\mathbb{Z}(\tilde{\delta})=pv$, for some non-zero $v\in S_\mathbb{Z}^n/M(\A,\bold{m})$. To conclude we just need to show that $v$ does not belong to the image of the map $\varphi_\mathbb{Z}$, and hence that $p$ is a zero divisor of the cokernel of $\varphi_\mathbb{Z}$, leading to a contradiction.

Suppose now that there exists $\sigma\in S_\mathbb{Z}^l$ such that $\varphi_\mathbb{Z}(\sigma)=v$. Since $p$ does not divide $\tilde{\delta}$, then $\tilde{\delta}-p\sigma\ne0$. Moreover, $\varphi_\mathbb{Z}(\tilde{\delta}-p\sigma)=0$, and hence $\tilde{\delta}-p\sigma\in\Ker(\varphi_\mathbb{Z})$. By construction, $\pi_p^l(\tilde{\delta}-p\sigma)=\delta$, and hence $\delta\in\pi_p^l(\Ker(\varphi_\mathbb{Z}))$, but this is impossible. Hence, $v$ does not belong to the image of the map $\varphi_\mathbb{Z}$, as claimed.
\end{proof}

Putting together Theorem~\ref{theo:fromPtozero} and Proposition~\ref{prop:zerodivtosurg}, we obtain the following result.

\begin{Theorem}\label{theo:fromPtozerobis} Let $(\A,\bold{m})$ be a multiarrangement in $\mathbb{Q}^l$. Let $p$
be a good prime number for $(\A,\bold{m})$ that is not a zero divisor of $\coker(\varphi_\mathbb{Z})$. If $(\A_p,\bold{m})$ is free in $\mathbb{F}_p^l$ with exponents $(e_1,\dots, e_l)$, 
then $(\A,\bold{m})$ is free in $\mathbb{Q}^l$ with exponents $(e_1,\dots, e_l)$.
\end{Theorem}

In Theorem \ref{theo:fromPtozerobis}, the assumption that the prime $p$ is not a zero divisor in the cokernel of the map $\varphi_\mathbb{Z}$ is fundamental. In fact we have the following.
\begin{Example} Consider the multiarrangement $(\A,\bold{m})$ in $\mathbb{Q}^3$ with defining polynomial $Q(\A,\bold{m})=x^2y^2(x-y)^2(x-z)^2(y-z)^2$. $(\A,\bold{m})$ is not free and $2$ is a zero divisor of in the cokernel of the map $\varphi_\mathbb{Z}$. In fact, 
$\varphi_\mathbb{Z}((x^2,y^2,z^2)^t)=(2x^2,2y^2,2(x^2-xy),2(x^2-xz),2(y^2-yz))^t$, however $(x^2,y^2,x^2-xy,x^2-xz,y^2-yz)^t$ is not in the image of $\varphi_\mathbb{Z}$. On the other hand, the multiarrangement $(\A_2,\bold{m})$ in $\mathbb{F}_2^3$ is free with exponents $(2,4,4)$.
\end{Example}

In general, given $M$ a finitely generated $S_\mathbb{Z}$-module, the number of zero divisor is infinite. However, if we restrict our attention to zero divisors that are prime numbers, we have the following.
\begin{Proposition}\label{prop:finitezerodivprimes} Let $M$ a finitely generated $S_\mathbb{Z}$-module. Then the number of distinct prime numbers that are zero divisors in $M$ is finite.
\end{Proposition}
\begin{proof}  By Theorem 14.4 of \cite{eisenbud}, there exists $a\in\mathbb{Z}\setminus\{0\}$ such that $M[a^{-1}]$ is a free $\mathbb{Z}[a^{-1}]$-module. This implies that 
the set of distinct prime numbers that are zero divisors in $M$ is included in the set of distinct prime numbers that divide $a$, that is finite by the unique factorization theorem.
\end{proof}

By applying Proposition \ref{prop:finitezerodivprimes} to the cokernel of the map $\varphi_\mathbb{Z}$, the number of prime numbers that are zero divisors in $\coker(\varphi_\mathbb{Z})$ is finite. Hence, putting together Corollary \ref{corol:fromchar0tocharP} 
and Theorem \ref{theo:fromPtozero}, we have the following.
\begin{Corollary}\label{corol:equivfreenesschar0p} Let $(\A,\bold{m})$ be a multiarrangement in $\mathbb{Q}^l$ and $p$ a large prime number. $(\A_p,\bold{m})$ is free in $\mathbb{F}_p^l$ with exponents $(e_1,\dots, e_l)$ if and only if
$(\A,\bold{m})$ is free in $\mathbb{Q}^l$ with exponents $(e_1,\dots, e_l)$.
\end{Corollary}

In \cite{palezzato2018free}, the authors studied the freeness of simple arrangements and
in Theorem~6.1, they proved the following.

\begin{Theorem}[\cite{palezzato2018free}]\label{theo:fromPtozeroOLD} Let $\A=(\A,\bold{1})$ be a simple central arrangement in $\mathbb{Q}^l$ and $J(\A)_\mathbb{Z}$ the ideal of $S_\mathbb{Z}$ generated by $Q(\A)=Q(\A,\bold{1})$ and its partial derivatives. Let $p$
be a good prime number for $\A$ that is not a zero divisor in $S_\mathbb{Z}/J(\A)_\mathbb{Z}$. If $\A_p$ is free in $\mathbb{F}_p^l$ with exponents $(e_1,\dots, e_l)$, then $\A$ is free in $\mathbb{Q}^l$ with exponents $(e_1,\dots, e_l)$.
\end{Theorem}

In the case of simple arrangements, Theorem~\ref{theo:fromPtozerobis} is exactly Theorem~\ref{theo:fromPtozeroOLD}. In fact we have the following.

\begin{Proposition} Let $\A=(\A,\bold{1})$ be a simple central arrangement in $\mathbb{Q}^l$. Assume that $p$ is a good prime for $\A$. Then $p$ is a zero divisor of $S_\mathbb{Z}/J(\A)_\mathbb{Z}$ if and only if $p$ is a zero divisor of $\coker(\varphi_\mathbb{Z})$.
\end{Proposition}
\begin{proof} Consider the map of $S_\mathbb{Z}$-modules $\psi\colon S_\mathbb{Z}^l\to S_\mathbb{Z}/\ideal{Q(\A)}S_\mathbb{Z}$ defined by $(g_1, \dots, g_l)^t\mapsto \sum_{i=1}^l g_i \partial Q(\A)/ \partial x_i$. By construction, the image of $\psi$ is $J(\A)_\mathbb{Z}S_\mathbb{Z}/\ideal{Q(\A)}S_\mathbb{Z}$, and hence $\coker(\psi)\cong S_\mathbb{Z}/J(\A)_\mathbb{Z}$.

Since $D(\A)=D(\A,\bold{1})=\{\delta\in \Der_{\mathbb{Q}^l}~|~ \delta(Q(\A)) \in \ideal{Q(\A)}S\}$, we have that $\Ker(\psi)\cong\Ker(\varphi_\mathbb{Z})$. By the first isomorphism theorem for modules, the image of $\psi$ and $\varphi_\mathbb{Z}$ are isomorphic and hence we have that $p$ is a zero divisor of $S_\mathbb{Z}/J(\A)_\mathbb{Z}$ if and only if $p$ is a zero divisor of $\coker(\varphi_\mathbb{Z})$.
\end{proof}

\bigskip
\paragraph{\textbf{Acknowledgements}} The author would like to thank E. Palezzato and G. Roehrle for many helpful discussions. During the preparation of this article the author was supported by JSPS Grant-in-Aid for Early-Career Scientists (19K14493). For details on how to perform the computations with (multi)arrangements, see \cite{palezzato2018hyperplane}.

\bibliography{bibliothesis}{}
\bibliographystyle{plain}

\end{document}